\documentclass{article}
\usepackage{amssymb}
\usepackage{amsthm}
\usepackage[curve,matrix,arrow]{xy}
\theoremstyle{plain}\newtheorem{Theorem}{Theorem}[section]
\theoremstyle{plain}
\theoremstyle{plain}
\theoremstyle{plain}\newtheorem{Lemma}[Theorem]{Lemma}
\theoremstyle{plain}\newtheorem{Proposition}[Theorem]{Proposition}
\theoremstyle{plain}
\theoremstyle{definition}
\theoremstyle{definition}\newtheorem{Example}[Theorem]{Example}
\theoremstyle{definition}
\theoremstyle{definition}\newtheorem{Remark}[Theorem]{Remark}
\theoremstyle{definition}
\theoremstyle{definition}

  \def\OG{{\mathcal{O}G}}  
  \def\OH{{\mathcal{O}H}}  
  \def\OP{{\mathcal{O}P}}
  \def\OQ{{\mathcal{O}Q}}
\def\CE{{\mathcal{E}}}  \def\OR{{\mathcal{O}R}}
\def\CF{{\mathcal{F}}}  
  \def\ON{{\mathcal{O}N}}

\def\CL{{\mathcal{L}}}

\def\CO{{\mathcal{O}}}

\def\CT{{\mathcal{T}}}

\def\R{{\mathbb R}}

\def\Aut{\mathrm{Aut}}                
\def\Br{\mathrm{Br}}             \def\ten{\otimes}

\def\chr{\mathrm{char}}
       
\def\dim{\mathrm{dim}}           
\def\End{\mathrm{End}}

\def\foc{\mathfrak{foc}}
          \def\GL{\mathrm{GL}}
\def\Hom{\mathrm{Hom}}

\def\ker{\mathrm{ker}}           
\def\Id{\mathrm{Id}}             \def\tenA{\otimes_A}
             \def\tenB{\otimes_B}
\def\Ind{\mathrm{Ind}}           
\def\Inn{\mathrm{Inn}}

           \def\tenL{\otimes_L}
           \def\tenO{\otimes_{\mathcal{O}}}
\def\mod{\mathrm{mod}}           

\def\op{\mathrm{op}}
\def\Out{\mathrm{Out}}
\def\Pic{\mathrm{Pic}}

\def\Stab{\mathrm{Stab}}
\def\tenOP{\otimes_{\mathcal{O}P}}
         \def\tenOQ{\otimes_{\mathcal{O}Q}}
\def\rk{\mathrm{rk}}      
\def\Res{\mathrm{Res}}           \def\tenOR{\otimes_{\mathcal{O}R}}

\title{Linear source invertible bimodules and Green correspondence} 
\author{Markus Linckelmann and Michael Livesey} 

\begin{document}

\maketitle

\begin{abstract}
We show that the Green correspondence induces an injective group 
homomorphism from the linear source Picard group $\CL(B)$
of a block $B$ of a finite group algebra to the linear source Picard 
group $\CL(C)$, where $C$ is the Brauer correspondent of $B$. 
This homomorphism maps the trivial source Picard group $\CT(B)$ 
to the trivial source Picard group $\CT(C)$.  
We show further that the endopermutation source Picard group $\CE(B)$ 
is bounded in terms of the defect groups of $B$ and that when $B$ has a normal defect group $\CE(B)=\CL(B)$. Finally we prove that the rank of any invertible $B$-bimodule is bounded by that of $B$.
\end{abstract}

\section{Introduction}

Let $p$ be a prime and $k$ a perfect field of characteristic $p$. We 
denote by  $\CO$  either a complete discrete valuation ring with 
maximal ideal $J(\CO)=$ $\pi\CO$ for some $\pi\in$ $\CO$, 
with residue  field $k$ and field of fractions $K$ of characteristic 
zero, or $\CO=k$.  We make the blanket assumption that $k$ and 
$K$ are large enough for the finite groups and their subgroups 
in the statements below. 

Let $A$ be an $\CO$-algebra. An $A$-$A$-bimodule $M$ is called
{\it invertible} if $M$ is finitely generated projective as a left
$A$-module, as a right $A$-module, and if there exits an 
$A$-$A$-bimodule $N$ which is finitely generated projective as a left 
and right $A$-module such that $M\tenA N\cong$ $A\cong$ $N\tenA M$ as 
$A$-$A$-bimodules.
The set of isomorphism classes of invertible $A$-$A$-bimodules is
a group, denoted $\Pic(A)$ and called the {\it Picard group of} $A$, 
where the product is induced by the tensor product over $A$. The
isomorphism class of the $A$-$A$-bimodule $A$ is the unit element 
of $\Pic(A)$.

Given a finite group $G$, a {\it block of $\OG$} is an indecomposable 
direct factor $B$ of $\OG$ as an algebra. Any such block $B$ determines 
a $p$-subgroup $P$ of $G$, called a {\it defect group of $B$}, uniquely
up to conjugation. Moreover, $B$ determines a block $C$ of $\CO N_G(P)$
with $P$ as a defect group, called the {\it Brauer correspondent of $B$}.
When regarded as an $\CO(G\times G)$-module, $B$ is a trivial source
module with vertices the $G\times G$-conjugates of the diagonal subgroup 
$\Delta P=$ $\{(u,u)\ |\ u\in P\}$ of $P\times P$. The 
Brauer correspondent $C$ is the Green correspondent of $B$ with
respect to the subgroup $N_G(P)\times N_G(P)$ of $G\times G$.
We denote by $\CL(B)$ the subgroup of $\Pic(B)$ of isomorphism classes
of invertible $B$-$B$-bimodules $X$ having a linear source (that is, a
source of $\CO$-rank $1$) for some (and hence any) vertex.
We denote by $\CT(B)$ the subgroup of $\CL(B)$ of isomorphism classes
of invertible $B$-$B$-bimodules $X$ having a trivial source for some
vertex. Note that if $\CO=k$, then $\CL(B)=\CT(B)$. In general, the
canonical surjection $B\to$ $k\tenO B$ induces an isomorphism
$\CT(B)\cong$ $\CT(k\tenO B)$ which extends to a surjective group
homomorphism $\CL(B)\to$ $\CT(k\tenO B)$. If $\chr(\CO)=0$, then
the kernel of this homomorphism is canonically isomorphic to 
$\Hom(P/\foc(\CF),\CO^\times)$; see 
\cite[Theorem 1.1, Remark 1.2.(d),(e)]{BoKL} for more details. 

\begin{Theorem} \label{PicGreencorr}
Let $G$ be a finite group, $B$ a block of $\OG$, and $P$ a defect group
of $B$. Set $N=$ $N_G(P)$ and denote by $C$ the block of $\ON$ which 
has $P$ as a defect group and which is  the Brauer correspondent of $B$. 
Let $X$ be a linear source invertible $B$-$B$-bimodule,
and let $Q$ be a vertex of $X$ contained in $P\times P$.
Then $N\times N$ contains the normaliser in $G\times G$ of the vertex $Q$
of $X$. Denote by $Y$ the $N\times N$-Green correspondent of $X$ with 
respect to $Q$.
Then $Y$ is a linear source invertible $C$-$C$-bimodule whose isomorphism 
class does not depend on the choice of $Q$ in  $P\times P$. Moreover, if 
$X$ has a trivial source, so does $Y$. 
The map $X\mapsto Y$ induces an injective group homomorphism
$$\CL(B) \to \CL(C)\ $$
which restricts to an injective group homomorphism
$$\CT(B) \to \CT(C)\ .$$
\end{Theorem}

The strategy is to translate this to a statement on the  source algebras
of $B$ and $C$, and then play this back to block algebras via the 
canonical Morita equivalences between blocks and their source algebras. 
By a result of Puig \cite[14.6]{Puigmodules}, a source 
algebra $A$ of $B$ contains canonically a source algebra $L$ of $C$.  

As a first step  we observe  that if $\alpha$ is an automorphism of $A$ 
which preserves $L$, then the $B$-$B$-bimodule corresponding to the 
invertible $A$-$A$-bimodule $A_\alpha$ is the Green correspondent 
of the $C$-$C$-bimodule corresponding to the invertible 
$L$-$L$-bimodule $L_\beta$, where $\beta$ is the automorphism of $L$
obtained from restricting $\alpha$.  See Proposition \ref{Greencorrbimod} 
for a precise statement and a proof, as well as the beginning of 
Section \ref{backgroundSection} for the notation. 

The second step  is to observe  that an element in $\CL(B)$, given by an 
invertible $B$-$B$-bimodule $X$,  corresponds via the canonical Morita 
equivalence, to an invertible $A$-$A$-bimodule 
of the form  $A_\alpha$ for some algebra automorphism $\alpha$ of $A$
which preserves the image of $\OP$ in $A$. 

The third and  key step is to show that $\alpha$  can be chosen in such 
a way that $\alpha$ preserves in addition the subalgebra $L$ of $A$. 
Such an $\alpha$ restricts therefore to an automorphism $\beta$ 
of $L$, and yields an invertible
$L$-$L$-bimodule $L_\beta$. By step one, the corresponding invertible 
$C$-$C$-bimodule $Y$ is then the Green correspondent of $X$, and the 
map $X\mapsto Y$ induces the group homomorphism $\CL(B)\to$ $\CL(C)$ 
as stated in the theorem. This third step proceeds in two stages - first
for the subgroup $\CT(B)$, and then for $\CL(B)$. This part of the proof 
relies significantly on the two papers \cite{Lifocal} and \cite{BoKL}. 

\begin{Example} \label{specialcase}
Consider  the special  case in Theorem \ref{PicGreencorr}
where $X$ is induced by a group automorphism $\alpha$  of $G$
which stabilises $B$. We use the same letter $\alpha$ for
the extension of $\alpha$ to an algebra automorphism of $\OG$.
Note that $\OG_\alpha$ is a permutation $\CO(G\times G)$-module.
Suppose that $\alpha$ stabilises $B$. Then the indecomposable direct 
summand $B_\alpha$ of $\OG_\alpha$  is a trivial source
$\CO(G\times G)$-module.  If $(P,e)$ is a maximal $B$-Brauer 
pair, then $(\alpha(P),\alpha(e))$ is a maximal $B$-Brauer pair as well, 
hence $G$-conjugate to $(P,e)$. After possibly composing $\alpha$ 
by a suitable chosen inner automorphism of 
$G$, we may assume that $\alpha$ stabilises $(P,e)$. 
Then $\alpha$ restricts to a group automorphism $\beta$ of $N=$ 
$N_G(P)$ which stabilises the Brauer correspondent $C$ of $B$.
The bimodule $B_\alpha$ represents an element in $\CT(B)$. Its 
$N\times N$-Green correspondent is the $C$-$C$-bimodule 
$C_\beta$, and this bimodule represents the image
in $\CT(C)$ under the homomorphism $\CT(B)\to$ $\CT(C)$ in
Theorem \ref{PicGreencorr}.
\end{Example}

By a result of Eisele \cite{EisPic},  if $\CO$ has characteristic zero, 
then $\Pic(B)$ is a finite group. We do not know, however, whether in 
that case the order of $\Pic(B)$ is determined `locally', that is, in 
terms of the defect groups of $B$. 
We show that there is a local bound, without any assumption on the 
characteristic of $\CO$,  for the order of the subgroup $\CE(B)$ of 
isomorphism classes of invertible $B$-$B$-bimodules $X$ having an 
endopermutation module as a source, for some vertex. 

\begin{Theorem} \label{Picbound}
Let $G$ be a finite group and $B$ a block algebra of $\OG$. Let 
$P$ be a defect group of $B$. Then the order of $\CE(B)$ is bounded
in terms of a bound which depends only on $P$.
\end{Theorem}

This will be proved as a consequence of \cite[Theorem 1.1]{BoKL}. 

\begin{Theorem} \label{Picbound2}
Let $A$ be an $\CO$-algebra which is free of finite rank as an 
$\CO$-module. Suppose that $k\tenO A$ is split and has a 
symmetric positive definite Cartan matrix. Then for every invertible 
$A$-$A$-bimodule $M$ we have $\rk_\CO(M)\leq$ $\rk_\CO(A)$.
Moreover, we have $\rk_\CO(M) =$ $\rk_\CO(A)$ if and only if
$M\cong$ $A_\alpha$ for some automorphism $\alpha$ of $A$. 
\end{Theorem}

This result applies in particular to any algebra $A$ which is Morita
equivalent to a block algebra of a finite group algebra over $\CO$ 
or $k$. We use this in the proof of the next result which is in turn 
used to show, in Example \ref{EPicExamples},  that Theorem 
\ref{PicGreencorr} does not hold with $\CL$ replaced by $\CE$.

\begin{Theorem} \label{LPicnormal}
Let $G$ be a finite group and $B$ a block of $\OG$ with a normal defect 
group. Then $\CE(B)=$ $\CL(B)$.
\end{Theorem}

\section{Background} \label{backgroundSection}

Let $A$, $B$ be $\CO$-algebras, and let $\alpha : A\to$ $B$ be an
algebra homomorphism. For any $B$-module $V$ we denote by
${_\alpha{V}}$ the $A$-module which is equal to $V$ as an
$\CO$-module, and on which $a\in$ $A$ acts as $\alpha(a)$. We use
the analogous notation for right modules and bimodules.

Any $A$-$A$-bimodule of the form ${A_\alpha}$ for some $\alpha\in$ 
$\Aut(A)$ is invertible, and we have $A\cong$ $A_\alpha$ as bimodules
if and only of $\alpha$ is inner. The map $\alpha\mapsto A_\alpha$
induces an injective group homomorphism $\Out(A)\to$ $\Pic(A)$.
This group homomorphism need not be surjective. 
An invertible $A$-$A$-bimodule $M$ is of the form ${A_\alpha}$ for some
$\alpha\in$ $\Aut(M)$ if and only if $M\cong$ $A$ as left 
$A$-modules, which is also equivalent to $M\cong$ $A$ as right
$A$-modules. 
See e.g. \cite[\S 55 A]{CR2} or \cite[Proposition 2.8.16]{LiBookI} 
for proofs and more details.

\begin{Lemma}[cf. {\cite[Lemma 2.4]{Lifocal}}] \label{autom-extend}
Let $A$ be an $\CO$-algebra and $L$ a subalgebra of $A$. Let $\alpha\in$ 
$\Aut(A)$ and let $\beta : L\to$ $A$ be an $\CO$-algebra homomorphism. 
The following are equivalent.

\begin{enumerate}
\item[{\rm (i)}]
There is an automorphism $\alpha'$ of $A$ which extends the map $\beta$ 
such that $\alpha$ and $\alpha'$ have the same image in $\Out(A)$.

\item[{\rm (ii)}]
There is an isomorphism of $A$-$L$-bimodules ${A_\beta }\cong$ 
${A_\alpha }$.

\item[{\rm (iii)}]
There is an isomorphism of $L$-$A$-bimodules ${_\beta A}\cong$ 
${_\alpha A}$.
\end{enumerate}
\end{Lemma}

\begin{Remark} \label{bimodswitch}
For $G$, $H$ finite groups, we switch without further comment between
$\CO(G\times H)$-modules and $\OG$-$\OH$-bimodules as follows: given
an $\OG$-$\OH$-bimodule $M$, we regard $M$ as an 
$\CO(G\times H)$-module (and vice versa) via $(x,y)\cdot m=$ 
$xmy^{-1}$, where $x\in$ $G$, 
$y\in$ $H$, and $m\in$ $M$. If $M$ is indecomposable as an 
$\OG$-$\OH$-bimodule, then $M$ is indecomposable as an 
$\CO(G\times H)$-module, hence has a vertex (in $G\times H$) and a 
source. If $Q$ is a subgroup of $G$, $R$ a subgroup of $H$, and $W$ an
$\CO(Q\times R)$-module, then with these identifications, we have an isomorphism
of $\CO(G\times H)$-modules (or equivalently of $\OG$-$\OH$-bimodules)
$$\Ind_{Q\times R}^{G\times H}(W) \cong \OG \tenOQ W\tenOR \OH $$
sending $(x,y)\ten w$ to $x\ten w\ten y^{-1}$, where $x\in G$, $y\in H$, and
$w\in W$. Thus if $M$ is a relatively $(Q\times R)$-projective 
$\CO(G\times H)$-module, then, as an $\OG$-$\OH$-bimodule,  $M$ is 
isomorphic to a direct summand of 
$$\OG\tenOQ W\tenOR \OH$$ 
for some $\OQ$-$\OR$-bimodule $W$.  Note further that if $B$ is a block of
$\OG$ with defect group $P$ and $C$ a block of $\OH$ with defect group
$Q$, then $B\tenO C^\op$ is a block of $\CO(G\times H)$
with defect group $P\times Q$, via the canonical algebra isomorphisms
$\CO(G\times H)\cong$ $\OG\tenO \OH\cong$ $\OG\tenO (\OH)^\op$. 
By~\cite[Lemma 2.3]{EatLivPic}, if $A$ and $L$ are source algebras of 
$B$ and $C$, respectively, then $A\tenO L^\op$ is a source algebra of 
$B\tenO C^\op$. 
\end{Remark} 

Let $G$ be a finite group and $B$ a block algebra of $\OG$ with a
defect group $P$. Recall our standing assumption that $K$ and $k$ are
splitting fields for the subgroups of $G$.
Choose a block idempotent $e$ of $kC_G(P)$ such that
$(P,e)$ is a maximal $B$-Brauer pair and a source idempotent
$i\in$ $B^P$ associated with $e$; that is, $i$ is a primitive 
idempotent in $B^P$ such that $\Br_P(i)e\neq$ $0$. Since $k$ is
assumed to be large enough, it follows that the choice of $(P,e)$ 
determines a (saturated) fusion system $\CF$ on $P$. In particular,
the group $\Out_\CF(P)\cong$ $N_G(P,e)/PC_G(P)$ is a $p'$-group,
and hence lifts uniquely up to conjugation by an element in $\Inn(P)$
to a $p'$-subgroup $E$ of $\Aut_\CF(P)\cong$ $N_G(P,e)/C_G(P)$.
The group $E$ is called the inertial quotient of $B$ (and depends on
the choices as just described).
As in  \cite[1.13]{AOV}, we denote by $\Aut(P,\CF)$ the subgroup of
$\Aut(P)$  consisting of all automorphisms of $P$ which stabilise $\CF$. 
In particular, the automorphisms in $\Aut(P,\CF)$ normalise the subgroup 
$\Aut_\CF(P)$ of $\Aut(P)$, and we set 
$$\Out(P,\CF)= \Aut(P,\CF)/\Aut_\CF(P)\ .$$ 
The algebra $A=$ $iBi$ is called a {\it source algebra of} $B$. If no
confusion arises, we identify $P$ with its image $iP=Pi$ in 
$A^\times$. Following~\cite{BoKL}, we set $\Aut_P(A)$ to be the 
group of algebra automorphisms of $A$ which fix $P$ elementwise, 
and by $\Out_P(A)$ the quotient of $\Aut_P(A)$ by the subgroup of 
inner automorphisms induced by conjugation with elements in 
$(A^P)^\times$.

\begin{Remark} \label{sourcealgebraRemark}
We will make use of the following standard facts on source algebras. 
With the notation above, by \cite[3.5]{Puigpoint} the $B$-$A$-bimodule 
$Bi$ and the  $A$-$B$-bimodule $iB$ induce a Morita  equivalence  
between $A$ and $B$. 
More precisely, the equivalence $iB\tenB- : \mod(B)\to$ $\mod(A)$ is 
isomorphic to the functor sending a $B$-module $U$ to the $A$-module
$iU$ and a $B$-module homomorphism $\varphi : U\to U'$ to the induced 
$A$-module homomorphism $iU\to iU'$ obtained from restricting $\varphi$ 
to $iU$. 

Following \cite[6.3]{LiKlein}, this Morita equivalence between $A$ 
and $B$ keeps track of  vertices and sources in the following sense: if 
$U$ is an indecomposable  $B$-module, then there exists a 
vertex-source pair $(Q,W)$ of $U$ such  that $Q\leq P$ and such that 
$W$ is isomorphic to a direct summand of  $\Res_Q(iU)$.  
In particular, a finitely generated $B$-module $U$ is a $p$-permutation 
$B$-module if and only if $iU$ is a $P$-permutation module. Since a
$p$-permutation $k\ten B$-module lifts uniquely, up to isomorphism, 
to a $p$-permutation $B$-module, it follows that a $P$-permutation
$k\tenO A$-module lifts uniquely, up to isomorphism, to a
$P$-permutation $A$-module. 

This applies to bimodules over block algebras via the Remark
\ref{bimodswitch}. Morita equivalences are compatible with tensor
products of algebras. In particular, there is a Morita equivalence between
$B\tenO B^\op$ and $A\tenO A^\op$ sending a $B$-$B$-bimodule $X$
to the $A$-$A$-bimodule $iXi$.  If $X$ is an  invertible 
$B$-$B$-bimodule, then $iXi$ is  an invertible $A$-$A$-bimodule, and 
the map $X\mapsto iXi$ induces a  canonical  group isomorphism 
$$\Pic(B)\cong \Pic(A)\ .$$
By the above remarks on vertices and sources applied to the block 
$B\tenO B^\op$ of $\CO(G\times G)$, as an $\CO(G\times G)$-module,
$X$ has a vertex $Q$ which is contained in $P\times P$, such that
the restriction to $P\times P$ of $iXi$ has an indecomposable direct 
summand with vertex $Q$. 

Following the notation in \cite{BoKL}, we denote by $\CE(B)$, $\CL(B)$,
$\CT(B)$ the subgroups of $\Pic(B)$ of isomorphism classes of
inevrtible $B$-$B$-bimodules with endopermutation, linear, trivial
sources, respectively. We denote by $\CE(A)$, $\CL(A)$, $\CT(A)$ their 
respective images in $\Pic(A)$ under the canonical isomorphism
$\Pic(B)\cong$ $\Pic(A)$. Again by the above remarks on vertices and
sources, this creates no conflict of notation if $A$ is a itself isomorphic
as an interior $P$-algebra to a block algebra of some (other) finite 
group with defect group $P$. 
See \cite[Section 6.4]{LiBookII} for an expository account on source 
algebras which includes the statements in this Remark.
\end{Remark}

The Brauer correspondent $C$ of $B$ has a source algebra $L$ of the
form $L=\CO_\tau(P\rtimes E)$ as an interior $P$-algebra, for some
$\tau\in$ $H^2(E, k^\times)$, where as above $E\cong$ $\Out_\CF(P)$ 
is the inertial quotient of $B$ and of $C$ determined by the choice of the
maximal $B$-Brauer pair $(P,e)$ (which is also a maximal $C$-Brauer
pair), and where we identify $k^\times$ with its canonical inverse image 
in $\CO^\times$. 

The fusion system of $C$ determined by the choice of $(P,e)$ is 
$N_\CF(P)=$ $\CF_P(P\rtimes E)$. Since $\Aut(P,\CF)$ is a subgroup
of $\Aut(P,N_\CF(P))$ and since $\Aut_{N_\CF(P)}(P)=$ 
$\Aut_\CF(P)$, it follows that $\Out(P,\CF)$ is a subgroup of
$\Out(P,N_\CF(P))$. 

By a result of Puig, there is a 
canonical embedding of interior $P$-algebras $L \to A$. We review in 
the following Proposition the construction of this embedding and some 
of its properties.

\begin{Proposition}[{cf. \cite[Theorem A]{Kuenormal}, 
\cite[Propositions 14.6, 14.9]{Puigmodules}, 
\cite[Proposition 4.10]{FanPuig}}]
\label{LtoA}
Let $G$ be a finite group, $B$ a block of $\OG$, and $(P,e)$ a
maximal $B$-Brauer pair, with associated inertial quotient $E$.
Let $C$ be the Brauer correspondent of $B$.
Denote by $\hat e$ the unique block 
idempotent of  $\CO C_G(P)$ which lifts $e$. The following hold. 

\begin{enumerate}
\item[{\rm (i)}]
Let $j$ be a primitive idempotent of $\CO C_G(P)\hat e$. Then $j$
remains primitive in $C^P$, and $j$ is a source idempotent both for
the block $C$ of $\CO N_G(P)$ as well as for $\CO N_G(P,e)\hat e$.
More precisely, the algebra  $L=$ $j\CO N_G(P)j$ is a source algebra
of $C$, and we have $L=$ $j \CO N_G(P,e) j$. There is $\tau\in$ 
$H^2(E, k^\times)$, inflated to $P\rtimes E$, such that 
$$L\cong \CO_\tau(P\rtimes E)$$
as interior $P$-algebras. 

\item[{\rm (ii)}]
Let $f$ be a primitive idempotent in $B^{N_G(P,e)}$ satisfying 
$\Br_P(f)e\neq 0$. Then $i=jf$ is a source idempotent in $B^P$ 
satisfying $\Br_P(i)e\neq 0$. Set $A=$ $iBi$. The idempotent $f$
commutes with $L$, and  multiplication by $f$ 
induces an injective homomorphism of interior $P$-algebras
$$L \to A$$
which is split injective as an $L$-$L$-bimodule homomorphism. Moreover,
every indecomposable direct $L$-$L$-bimodule summand of $A$ in a 
complement of $L$ is relatively projective, as an 
$\CO (P\times P)$-module, with respect to a twisted 
diagonal subgroup of $P\times P$ of order strictly smaller than $|P|$.

\item[{\rm (iii)}]
As an $A$-$A$-bimodule, $A$ is isomorphic to a direct summand of
$A\ten_L A$, and every other indecomposable direct $A$-$A$-bimodule
summand of $A\ten_L A$ is relatively projective, as an 
$\CO (P\times P)$-module, with respect to a 
twisted diagonal subgroup of $P\times P$ of order strictly smaller 
than $|P|$.

\item[{\rm (iv)}]
The map sending $\zeta\in$ $\Hom(E,k^\times)$ to the linear
endomorphism of $L$ given by the assignment $uy\mapsto$ 
$\zeta(y)uy$, where $u\in$ $P$ and $y\in$ $E$, and where we identify 
$L=$ $\CO_\tau(P\rtimes E)$ induces a group isomorphism
$$\Hom(E,k^\times) \cong \Out_P(L)$$

\item[{\rm (v)}]
The map sending an automorphism $\alpha$ of $A$ which fixes
$P$ elementwise and stabilises $L$ to the restriction of $\alpha$
to $L$ induces an injective group homomorphism 
$$\Out_P(A)\to \Out_P(L) \cong \Hom(E,k^\times)\ .$$

\end{enumerate}
\end{Proposition}

Proofs of the statements in Proposition \ref{LtoA} can be found in the
expository account of this material in \cite[Theorem 6.14.1]{LiBookII}, 
\cite[Theorem 6.7.4]{LiBookII}, \cite[Theorem 6.15.1]{LiBookII},
and \cite[Lemma 6.16.2]{LiBookII}.
We record the following elementary group theoretic observation.

\begin{Lemma} \label{Qnormaliser}
Let $G$ be a finite group and $P$ a subgroup. Let $Q$ be a subgroup of
$P\times P$. Suppose that the two canonical projections $P\times P\to P$
both map $Q$ onto $P$. The following hold.

\begin{enumerate}
\item[{\rm (i)}]
If $(x,y)\in G\times G$ such that ${^{(x,y)}{Q}}\leq P\times P$, then
$(x,y)\in N_G(P)\times N_G(P)$.

\item[{\rm (ii)}] 
We have $N_{G\times G}(Q)\leq N_G(P)\times N_G(P)=
N_{G\times G}(P\times P)\ .$
\end{enumerate}
\end{Lemma}

\begin{proof}
Let $(x,y)\in$ $G\times G$ such that ${^{(x,y)}{Q}}\leq P\times P$.
Let $u\in$ $P$. Since the first projection $P\times P\to P$ maps $Q$ 
onto $P$, it follows that there is $v\in P$ such that $(u,v)\in$ $Q$. Then 
${^{(x,y)}{(u,v)}}\in$ $P\times P$. In particular, ${^x{u}}\in P$. Thus $x\in$
$N_G(P)$. The same argument yields $y\in$ $N_G(P)$, and hence
$(x,y)\in$ $N_G(P)\times N_G(P)$. This shows (i), and (ii) follows 
immediately from (i).
\end{proof}

\begin{Remark} \label{GreencorrRem}
Let $G$ be a finite group, $P$ a $p$-subgroup, and $X$ an indecomposable
$\CO(G\times G)$-module with a vertex $Q$ contained in $P\times P$ such 
that the two canonical projections $P\times P\to P$ map $Q$ onto $P$. By
Lemma \ref{Qnormaliser} (ii),  the Green correspondence yields, up to 
isomorphism,  a unique  indecomposable direct summand $f(X)$ of the 
$\CO(N_G(P)\times N_G(P))$-module 
$\Res^{G\times G}_{N_G(P)\times N_G(P)}(X)$ with vertex $Q$ and a source 
which remains a source of $X$. Since any two vertices of $X$ are 
$G\times G$-conjugate, it follows from Lemma \ref{Qnormaliser} (i) that
the isomorphism class of $f(X)$ does not depend on the choice of a vertex 
$Q$ of $X$ in $P\times P$.  
\end{Remark}

\begin{Lemma} \label{PicAvertices}
Let $A$ be a source algebra of a block $B$ of a finite group algebra
$\OG$ with defect group $P$. Let $M$ be an invertible $A$-$A$-bimodule, 
and let $X$ be an invertible $B$-$B$-bimodule. The following hold.

\begin{enumerate}
\item[{\rm (i)}]
$M$ remains indecomposable as an $A$-$\OP$-bimodule and as an 
$\OP$-$A$-bimodule.

\item[{\rm (ii)}]
As an $\CO(P\times P)$-module, $M$ has an indecomposable direct 
summand with a vertex $Q$ such that both canonical projections 
$P\times P\to P$ map $Q$ onto $P$. In particular, $Q$ has order at 
least $|P|$.

\item[{\rm (iii)}]
As an $\CO(G\times G)$-module, $X$ has a vertex $Q$ contained in 
$P\times P$, and any such vertex $Q$ has the property that both
canonical projections $P\times P\to P$ map $Q$ onto $P$. In particular,
the vertices of $X$ have order at least $|P|$.

\end{enumerate}
\end{Lemma}

\begin{proof}
Since $M$ is an invertible bimodule, by Morita's theorem, we have an
algebra isomorphism $A\cong$ $\End_{A^\op}(M)$ sending $c\in$ $A$ to
the right $A$-endomorphism of $M$ given by left multiplication by $c$ 
on $M$. This restricts to an algebra isomorphism $A^P\cong$ 
$\End_{\OP\tenO A^\op}(M)$. Now $1_A=i$ is primitive in $B^P$, hence
$A^P$ is local, and thus so is $\End_{\OP\tenO A^\op}(M)$, implying the
statement (i). We prove next (iii). Suppose that the first projection
$P\times P\to$ $P$ maps $Q$ onto the subgroup $R$ of $P$. Then
$Q$ is contained in $R\times P$, hence in $R\times G$. It follows that $X$ 
is relatively $(R\times G)$-projective. Thus $X$ is isomorphic to a direct 
summand of $\OG\tenOR X$ as a bimodule. Let $U$ be a $B$-module. 
Then $U\cong$ $X\tenB V$ for some $B$-module $V$ because $X$ is an 
invertible bimodule. Thus $U\cong$ $X\tenB V$ is isomorphic to a direct 
summand of $\OG\tenOR X\tenB V\cong$ $\Ind^G_R(U)$. This shows that 
every  $B$-module is relatively $R$-projective, which forces $R=P$. A 
similar argument, using right $B$-modules, shows that the second 
projection $P\times P\to$ $P$ maps $Q$ onto $P$. This implies  (iii).
Statement  (ii) follows from (iii) by the Remark 
\ref{sourcealgebraRemark}.
\end{proof}

\begin{Remark}
The statements (ii), (iii) in Lemma \ref{PicAvertices} hold more 
generally if $M$, $X$ induce a stable equivalence of Morita  type; 
see \cite[Section 6]{Puigbook}.
\end{Remark}

\begin{Lemma}[{cf. \cite[Corollary 2.4.5]{LiBookI}}]  
\label{diagLemma}
Let $P$ be a finite group and $\varphi\in$ $\Aut(P)$. Set
$Q=$ $\{(\varphi(u), u)\ |\ u\in P\}$. We have
an isomorphism of $\CO(P\times P)$-modules
$$\Ind_Q^{P\times P}(\CO) \cong (\OP)_\varphi$$
sending $(x,y)\ten 1$ to $x\varphi(y^{-1})$, for all $x$, $y\in$ $P$.
\end{Lemma}

\begin{proof} 
This is easily verified directly. Note that this is the special 
case of \cite[Corollary 2.4.5]{LiBookI} applied to 
$G=$ $P\rtimes \langle \varphi\rangle $, $H=L=P$ and $x=$ $\varphi$. 
\end{proof}

\begin{Lemma} \label{ALAut}
Let $\alpha$ be an algebra automorphism of $A$ which preserves
$L$. Denote by $\beta$ the algebra automorphism of $L$ obtained
from restricting $\alpha$ to $L$. 

\begin{enumerate}
\item[{\rm (i)}]
The class of $\beta$ in $\Out(L)$ is uniquely determined by the class 
of $\alpha$ in $\Out(A)$.

\item[{\rm (ii)}]
If $\alpha$ is an inner automorphism of $A$, then $\beta$ is an inner 
automorphism of $L$.

\item[{\rm (iii)}]
If $c\in$ $A^\times$ satisfies $cLc^{-1}=L$, then there exists 
$d\in L^\times$ such that $cd^{-1}$ centralises $L$.
\end{enumerate}
\end{Lemma}

\begin{proof}
We first prove (iii). Let $c\in$ $A^\times$
such that $cLc^{-1}=L$. Then $Lc$ is an $L$-$L$-bimodule summand
of $A$. Note that $L$ and $Lc$ have the same $\CO$-rank $r$, and that
$\frac{r}{|P|}=$ $|E|$ is prime to $p$. It follows from Proposition 
\ref{LtoA} (ii) that $L$ is up to isomorphism the unique 
$L$-$L$-bimodule summand of $A$ with this property, and hence
$L\cong$ $Lc$ as $L$-$L$-bimodules. This implies that conjugation
by $c$ on $L$ induces an inner automorphism of $L$, given by
an element $d\in$ $L^\times$. Then $cd^{-1}$ acts as the identity
on $L$. This proves (iii). 
The statements (ii) and (iii) are clearly equivalent. In order to show (i),
let $\alpha$, $\alpha'$ two automorphisms which
preserve $L$ and which represent the same class in $\Out(A)$. Thus 
there exists $c\in$ $A^\times$ such that $\alpha'(a)=$ $c\alpha(a) c^{-1}$
for all $a\in$ $A$. Since $\alpha$, $\alpha'$ preserve $L$, it follows 
that conjugation by $c$ preserves $L$. By (ii), conjugation by $c$ induces
an inner automorphism of $L$, and hence the restrictions to $L$ of
$\alpha$, $\alpha'$ belong to the same class in $\Out(L)$. 
The result follows.
\end{proof}

\begin{Lemma} \label{fusionLemma}
Let $A$ be a source algebra of a block with defect group $P$ and
fusion system $\CF$ on $P$. Let $\psi\in$ $\Aut(P)$. There is an
isomorphism of $A$-$\OP$-bimodules $A\cong$ $A_\psi$ if
and only if $\psi\in$ $\Aut_\CF(P)$.
\end{Lemma}

\begin{proof}
This is the special case of the equivalence of the statements (i) and (iii) 
in \cite[Theorem 8.7.4]{LiBookII}, applied to $Q=R=P$ and $m=n=1_A$.
\end{proof}

For further results  detecting fusion in source algebras
see \cite{Puigloc}, or also \cite[Section 8.7]{LiBookII}.

\section{Source algebra automorphisms and Green correspondence}
\label{SourceAut}

We use the notation and facts reviewed in Proposition \ref{LtoA}.
In particular, $A=$ $iBi$ and $L=jCj$ are source algebras of the block
$B$ of $\OG$ and its Brauer correspondent $C$, respectively, both 
associated with  a maximal Brauer pair $(P,e)$, and chosen 
such that multiplication by a primitive idempotent $f$ in $B^{N_G(P,e)}$ 
satisfying $\Br_p(f)e\neq$ $0$ induces an embedding $L\to A$ as
interior $P$-algebras. In particular, $i=jf$. This embedding is split 
as a homomorphism of $L$-$L$-bimodules. We set $N=$ $N_G(P)$. 
We keep this notation throughout this section. 

The following Proposition describes the Green correspondence at the 
source algebra level for certain invertible bimodules induced by 
automorphisms.

\begin{Proposition} \label{Greencorrbimod}
Let $\beta$ be an algebra automorphism of $L$ which extends to an 
algebra automorphism $\alpha$ of $A$ through the canonical embedding
$L\to A$. Then the $B$-$B$-bimodule
$$X = \OG i_\alpha \tenA i\OG$$
is invertible. As an $\CO(G\times G)$-module, $X$ has a vertex $Q$
contained in $P\times P$. The $\CO(N\times N)$-Green correspondent 
of $X$ with respect to $Q$  is isomorphic to 
$$Y = \ON j_\beta \tenL j\ON\ .$$
\end{Proposition}

\begin{proof}
The $A$-$A$-bimodule $A_\alpha$ is obviously invertible, and hence
so is the $B$-$B$-bimodule $X$, since $X$ is the image of $A_\alpha$
under the canonical Morita equivalence between $A\tenO A^\op$ and
$B\tenO B^\op$. Similarly, $L_\beta$ and $Y$ are invertible bimodules.
By Lemma \ref{PicAvertices}, $Y$ has a vertex $Q$ contained in $P\times P$
such that both canonical projections $P\times P\to$ $P$ map $Q$ onto
$P$. By Lemma \ref{Qnormaliser}, we have $N_{G\times G}(Q)\leq$
$N\times N$, so $Y$ has a well-defined Green correspondent. In
order to show that $X$ is this Green correspondent, we start by showing
that $X$ is isomorphic to a direct summand of 
$\Ind_{N\times N}^{G\times G}(Y)$. Rewrite
$$\Ind_{N\times N}^{G\times G}(Y) = \OG j_\beta \tenL j \OG\ .$$
Decompose $j=i + (j-i)$; since $i=jf$, this is an orthogonal decomposition
of $j$ into two idempotents both of which commute with $L$; that is,
$\OG i_\beta \tenL i \OG$ is isomorphic to a direct summand
of $\OG j_\beta \tenL j \OG$. We show that $X$ is isomorphic
to a direct summand of $\OG i_\beta \tenL i\OG$. Multiplying both sides
by $i$, this is equivalent to showing that $A_\alpha$ is isomorphic to
a direct summand of $A_\beta \tenL A$. Now $A$ is isomorphic to a
direct summand of $A\tenL A$ (cf. Proposition \ref{LtoA}).
Tensoring on the left with $A_\alpha$ shows that $A_\alpha$ is isomorphic 
to a direct summand of $A_\alpha\tenL A=$ $A_\beta\tenL A$, where the
last equality holds since $\alpha$ extends $\beta$. This shows that
$X$ is indeed isomorphic to a direct summand of 
$\Ind^{G\times G}_{N\times N}(Y)$, and therefore $X$ has a subgroup of 
$Q$ as a vertex.
In order to show that $X$ is the Green correspondent of $Y$, we need to
show that $Q$ is a vertex of $X$. It suffices to show that $Y$ is 
isomorphic to a direct summand of $\Res^{G\times G}_{N\times N}(X)=$
$\OG i_\alpha\tenA i\OG$. Thus it suffices to show that $L_\beta$
is isomorphic to a direct summand of $j\OG i_\alpha\tenA i\OG j$.
Using as before the decomposition $j=i+(j-i)$, it suffices to show that
$L_\beta$ is isomorphic to a direct summand of $A_\alpha$,
as an $L$-$L$-bimodule. By Proposition \ref{LtoA} (ii), $L$ is isomorphic 
to a direct summand of $A$. The claim now follows by tensoring on the 
right with $L_\beta$ and noting that $A\tenL L_\beta\cong A_\alpha$ 
as $L$-$L$-bimodules.
\end{proof}

As before, we denote by $\CE(B)$, $\CL(B)$, $\CT(B)$ the
subgroups of $\Pic(B)$ represented by invertible bimodules whose sources 
are endopermutation, linear, or trivial, respectively. We denote by
$\CE(A)$, $\CL(A)$, $\CT(A)$ the subgroups of $\Pic(A)$ which correspond
to $\CE(B)$, $\CL(B)$, $\CT(B)$, respectively, under the canonical group 
isomorphism $\Pic(B)\cong$ $\Pic(A)$ (cf. Remark \ref{sourcealgebraRemark}).
Summarising special cases of results in \cite[6.7]{Puigbook},
\cite[Section 2]{BoKL}, the bimodules representing elements in $\CT(A)$ 
can be described as follows.

\begin{Proposition}\label{TA-Prop3}
An invertible $A$-$A$-bimodule $M$ represents an element in $\CT(A)$ 
if and only if $M$ is isomorphic to a direct summand of an 
$A$-$A$-bimodule of the form 
$$A_\varphi \tenOP A$$
for some $\varphi\in$ $\Aut(P,\CF)$. Moreover, the class of $\varphi$
in $\Out(P,\CF)$ is uniquely determined by the isomorphism class of
$M$, and the map $M\mapsto \varphi$ induces a group
homomorphism $\CT(A) \to \Out(P,\CF)$.  This group homomorphism
corresponds to the group homomorphism $\CT(B)\to$ $\Out(P,\CF)$
in \cite[Theorem 1.1.(ii)]{BoKL} through the canonical isomorphism
$\CT(B)\cong$ $\CT(A)$.
\end{Proposition}

\begin{proof}
The fact that the elements in $\CT(A)$ are represented by bimodules
as stated follows from the reformulation \cite[Theorem 2.4]{BoKL}
of results of Puig \cite[7.6]{Puigbook} together with the canonical 
Morita equivalence between $B$ and $A$. The statement on the uniqueness 
of the class of $\varphi$ in $\Out(P,\CF)$ follows from 
\cite[Lemma 2.7]{BoKL}. The fact that this yields a group homomorphism 
$\CT(A)\to$ $\Out(P,\CF)$ follows from \cite[Lemma 2.6]{BoKL}. By 
construction, this yields the group homomorphism $\CT(B)\to$ 
$\Out(P,\CF)$ in \cite[Theorem 1.1. (ii)]{BoKL} when precomposed with 
the canonical isomorphism $\CT(B)\cong$ $\CT(A)$. 
\end{proof}

The following characterisation of $A$-$A$-bimodules in Proposition 
\ref{TA-Prop1} (i) below representing elements in $\CT(A)$ is essentially a 
reformulation of work  of L. L. Scott \cite{Scottnotes} and L. Puig 
\cite{Puigbook}, where it is shown that Morita equivalences between 
block algebras given by $p$-permutation bimodules are induced by 
source algebra isomorphisms. As in Proposition \ref{TA-Prop3}, the 
homomorphism $\CT(A)\to \Out(P,\CF)$ in Proposition 
\ref{TA-Prop1} (ii) corresponds to the one at the bottom of the 
diagram in \cite[Theorem 1.1]{BoKL}. 

\begin{Proposition}\label{TA-Prop1}
With the notation above, the following hold.

\begin{enumerate}
\item[{\rm (i)}]
An invertible $A$-$A$-bimodule $M$ represents an element in $\CT(A)$
if and only if $M\cong$ $A_\alpha$ for some $\CO$-algebra 
automorphism $\alpha$ of $A$ which preserves the image of $P$ in
$A^\times$. In particular, $\CT(A)$ is a subgroup of the image
of $\Out(A)$ in $\Pic(A)$. 

\item[{\rm (ii)}]
Let $\varphi\in$ $\Aut(P)$ and let $\alpha$ be an $\CO$-algebra 
automorphism of $A$ which extends $\varphi$.
Then $\varphi\in$ $\Aut(P,\CF)$, and the map
$\alpha\mapsto\varphi$ induces a group homomorphism 
$$\CT(A)\to \Out(P,\CF)$$ 
with kernel $\Out_P(A)$.
\end{enumerate}
\end{Proposition}

\begin{proof}
Note that $A$ is a permutation $\OP$-$\OP$-bimodule. Thus if
$\alpha\in$ $\Aut(A)$ preserves the image of $P$ in $A^\times$, then
$A_\alpha$ is again a permutation $\OP$-$\OP$-bimodule. Therefore
$A_\alpha$ represents in that case an element in $\CT(A)$. 
For the converse, let $M$ be an invertible $A$-$A$-bimodule which
represents an element in $\CT(A)$. By Proposition \ref{TA-Prop3},
there is $\varphi\in$ $\Aut(P,\CF)$ such that $M$ is isomorphic to a 
direct summand of $A_\varphi \tenOP  A$. By Lemma \ref{PicAvertices},
the restriction of $M$ as an $A$-$\OP$-bimodule remains indecomposable.
Thus, using Krull-Schmidt, as an $A$-$\OP$-bimodule, $M$ is 
isomorphic to a direct summand of $A_\varphi \tenOP W$ for some
indecomposable direct summand $W$ of $A$ as an $\OP$-$\OP$-bimodule.
By \cite[Theorem 8.7.1]{LiBookII}, we have $W\cong$ 
$OP_\tau \tenOQ OP$ for some subgroup $Q$ of $P$ and some 
$\tau\in$ $\Hom_\CF(Q,P)$. By Lemma \ref{PicAvertices} (ii), we have $Q=P$,
and hence $\tau\in$ $\Aut_\CF(P)$ and $W=$ $\OP_\tau$. Thus $M$ 
is isomorphic to a direct summand of $A_\varphi \tenOP \OP_\tau\cong$
$A_{\varphi\circ\tau} \cong$ $A_{{}^\varphi\tau\circ\varphi} \cong$ 
$A_\varphi$, where the last isomorphism uses Lemma 
\ref{fusionLemma} and the fact that ${}^\varphi\tau\in$ $\Aut_\CF(P)$.   
But then $M$, as an $A$-$\OP$-module, is isomorphic to $A_\varphi$, 
since, by Lemma \ref{PicAvertices} (i), this 
module is indecomposable. In particular, $M\cong$ $A$ as a left 
$A$-module. Thus $M\cong$ $A_\alpha$ for some $\alpha\in$ $\Aut(A)$. 
By Lemma \ref{autom-extend} we can choose $\alpha$ to extend 
$\varphi$.

The fact that we have a group homomorphism $\CT(A)\to \Out(P,\CF)$ 
follows from Proposition \ref{TA-Prop3} and that its kernel is 
$\Out_P(A)$ from identifying it with the corresponding homomorphism 
in \cite[Theorem 1.1]{BoKL}.
\end{proof}

The next result shows that in the situation of Proposition \ref{TA-Prop1}
(ii)  it is possible to choose $\alpha$ in such a way that 
it preserves the subalgebra $L=$ $\CO_\tau(P\rtimes E)$. 

\begin{Proposition} \label{TA-Prop2}
With the notation above, let $\varphi\in$ $\Aut(P)$ such that $\varphi$
extends to an $\CO$-algebra automorphism $\alpha$ of $A$. Then 
$\varphi$  extends to an $\CO$-algebra automorphism $\alpha'$ of $A$ 
such that the images of $\alpha$ and $\alpha'$ in $\Out(A)$ are equal 
and such that $\alpha'$ preserves the subalgebra $L$.
The correspondence $\alpha\mapsto$ $\alpha'|_{L}$
induces an injective  group homomorphism
$\rho : \CT(A) \to \CT(L)$, and we have a commutative diagram of
finite groups with exact rows of the form
$$\xymatrix{1 \ar[r] & \Out_P(A) \ar[r] \ar[d] 
& \CT(A) \ar[r] \ar[d]^{\rho} & \Out(P,\CF) \ar[d] \\
1 \ar[r] & \Hom(E,k^\times) \ar[r] & \CT(L) \ar[r] & \Out(P,N_\CF(P))}$$
where the leftmost vertical map is from Proposition \ref{LtoA} (v), 
after identifying $\Hom(E,k^\times)$ with $\Out_P(L)$ via Proposition 
\ref{LtoA} (iv), the rightmost horizontal arrows are those from 
Proposition \ref{TA-Prop1}, and the right vertical map is the inclusion.
\end{Proposition}

\begin{proof}
In order to prove the first statement, we need to show that $\alpha(L)$
is conjugate to $L$ via an element $w$ in $(A^P)^\times$. 
This proof is based on a `Maschke type' argument, constructing
$w$ explicitly. This is a well-known strategy;  
see e.g. \cite[Remark 4.4]{HeKi},  \cite[Proposition 4]{LiSchroll}. 

Note that any inner automorphism of $P$ extends trivially  to an algebra 
automorphism of $A$.
Since $\varphi$ extends to an algebra automorphism $\alpha$ of $A$, it 
follows that any $\varphi'\in$ $\Aut(P)$ representing the same class as 
$\varphi$ in $\Out(P)$ extends to an algebra automorphism of $A$ 
representing the same class as $\alpha$ in $\Out(A)$. Therefore, in 
order to prove Proposition \ref{TA-Prop2}, we may replace $\varphi$ by 
any automorphism of $P$ representing the same class as $\varphi$ in 
$\Out(P)$.

We identify $P\rtimes E$ as a subset of $L=\CO_\tau(P\rtimes E)$, hence
of $A$. Note that this is a subset of $A^\times$, but not a subgroup,
because of the twist of the multiplication by $\tau$. In particular, the 
inverse $x^{-1}$ in the group $P\rtimes E$ of an element $x\in$ 
$P\rtimes E$ is in general different from the inverse of $x$ in the 
algebra $L$. More precisely, the inverses of $x$ in the group $P\rtimes E$
and in the algebra $L$ differ by a scalar. 

For group elements  $x$, $y\in$ $P\rtimes E$, we denote by $xy$ the 
product in the group $P\rtimes E$, and by $x\cdot y$ the product in the 
algebra $L$; that is, we have 
$$x\cdot y = \tau(x,y)xy\ .$$
We denote by ${^yx}$ the conjugate of $x$ by $y$ in the group 
$P\rtimes E$. By the above, this differs by a scalar from the conjugate of 
$x$ by $y$ in $A^\times$. 

Let $\varphi\in$ $\Aut(P)$ and $\alpha\in$ $\Aut(A)$ such that
$\alpha$ extends $\varphi$. By Proposition \ref{TA-Prop1} we have
$\varphi\in$ $\Aut(P,\CF)$. In particular, $\varphi$ normalises
the group $\Aut_\CF(P)=$ $\Inn(P)\cdot E$. Then 
$\varphi\circ E \circ\varphi^{-1}$ is a complement of $\Inn(P)$ in $\Inn(P)\cdot E$,
so conjugate to $E$ by an element in $\Inn(P)$ by the Schur-Zassenhaus
theorem. That is, after possibly replacing $\varphi$ by another 
representative in $\Aut(P)$ of the class of $\varphi$ in $\Out(P)$, we 
may assume that $\varphi$ normalises the subgroup $E$ of $\Aut(P)$.

Let $y\in$ $E$ (regarded as an automorphism of $P$). 
Since $\varphi$ normalises $E$, there is an element $\psi(y)\in$ $E$ 
such that
$$\varphi\circ y\circ\varphi^{-1} = \psi(y)\ .$$
That is,  $\psi$ is the group  automorphism of $E$ induced
by conjugation with $\varphi$ in $\Aut(P)$.  

In what follows we denote by $\psi(y)^{-1}$ the inverse of $\psi(y)$ 
in the subalgebra $L$ of $A$; by the above, this may differ from the 
group theoretic inverse of $\psi(y)$ in $E$  by a scalar in $\CO^\times$. 
The elements $\alpha(y)$ and $\psi(y)$ in $A^\times$ act in the same 
way on the image of $P$ in $A$ up to scalars in $\CO^\times$. That is,
conjugation by $\alpha(y)\psi(y)^{-1}$ in $A^\times$ sends $u\in$ $P$ 
to $\zeta(u)u$ for some scalar $\zeta(u)\in$ $\CO^\times$. The map 
$u\mapsto \zeta(u)$ is then a group homomorphism from $P$ to 
$\CO^\times$. It follows from \cite[Lemma 3.9]{Lifocal} that $\zeta(u)=1$ 
for all $u\in $ $P$. This shows that  $\alpha(y)\psi(y)^{-1}$ belongs to 
$(A^P)^\times$. 
Since conjugation by the elements $\alpha(y)$, $\psi(y)$ in $A^\times$ 
preserves $\OP$, these conjugations also preserve the centraliser 
$A^P$ of $\OP$ in $A$. In other words, $\alpha(y)$ and $\psi(y)$ 
normalise the subgroups $(A^P)^\times$ and $1+J(A^P)$ of 
$A^\times$.

Since $k$ is perfect, we have a canonical group isomorphism
$\CO^\times\cong$ $k^\times\times(1+J(\CO))$. 
Now $A^P$ is a local algebra, so $(A^P)^\times=$ $k^\times(1+J(A^P))$, 
or equivalently, every element in $(A^P)^\times$ can be 
written uniquely in the form $\lambda\cdot 1_A+r$ for some 
$\lambda\in$ $k^\times$ (with $k^\times$ identified to its canonical 
preimage in $\CO^\times$) and some $r\in$ $J(A^P)$. Thus 
$$\alpha(y)\psi(y)^{-1} = \lambda_y + r_y$$
for a uniquely determined $\lambda_y\in$ $k^\times$ and $r_y\in$ 
$J(A^P)$.  It follows that $\lambda_y^{-1}\alpha(y)\psi(y)^{-1}\in$ 
$1+J(A^P)$.  Set 
$$w = \frac{1}{|E|}\sum_{y\in E}\ 
\lambda_y^{-1}\alpha(y)\psi(y)^{-1}\ .$$
This is well defined since $|E|$ is prime to $p$. 
By construction, we have $w\in$ $1+J(A^P)$, so in particular, $w$ is
invertible in $A^P$, and conjugation by $w$ fixes the elements of $P$,
hence preserves $\OP$.  
Therefore, in order to show that $\alpha(L)=$ $wLw^{-1}$, it 
suffices to show that for any $y\in$ $E$, the element $\alpha(y)$ is
a scalar multiple of the conjugate $w\psi(y)w^{-1}$.  More
precisely, we are going to show that
$$\alpha(y)w= \lambda_y w\psi(y)\ .$$ 
For any further element $x\in$ $E$, applying $\alpha$ to the
equation $y\cdot x=$ $\tau(y,x) yx$ yields 
$$\alpha(y) \alpha(x)=\tau(y,x)\alpha(yx)\ .$$ 
Similarly, we have 
$$\psi(y)\cdot \psi(x) = \tau(\psi(y),\psi(x)) \psi(yx)\ .$$
We show next that the $2$-cocycles $\tau$ and
$\tau(\psi(-), \psi(-))$ in $Z^2(E,k^\times)$ represent the same class, 
via the $1$-cochain $y\mapsto \lambda_y$. By construction, $\alpha(y)$ 
and  $\lambda_y\psi(y)$ differ by an element in $1+J(A^P)$. 
Calculating modulo $1+J(A^P)$ in the two previous equations yields
$$\tau(\psi(y),\psi(x)) = \lambda_y^{-1}\lambda_x^{-1}\lambda_{yx}
\tau(y,x)\ .$$
In other words, the class of $\tau$ is stable under $\psi$.

Using these equations,  we have
$$\alpha(y) w = 
\frac{1}{|E|} \sum_{x\in E} \lambda_x^{-1}\alpha(y) \alpha(x)\psi(x)^{-1} =
\frac{1}{|E|} \sum_{x\in E} \lambda_x^{-1} \tau(y,x) \alpha(yx) 
\tau(\psi(y),\psi(x))^{-1} \psi(yx)^{-1}\psi(y) =$$
$$\frac{1}{|E|} \sum_{x\in E} \lambda_x^{-1} \tau(y,x) \alpha(yx) 
\lambda_x\lambda_y\lambda_{yx}^{-1} \tau((y, x)^{-1} \psi(yx)^{-1}\psi(y) =$$
$$\frac{1}{|E|} \sum_{x\in E} \lambda_y\lambda_{yx}^{-1} 
\alpha(yx)\psi(yx)^{-1} \psi(y) = \lambda_y w\psi(y)\ .$$
This shows that $\alpha(L)=$ $wLw^{-1}$. Thus setting 
$$\alpha'= c_{w^{-1}}\circ\alpha\ , $$
where here $c_{w^{-1}}$ is conjugation by $w^{-1}$ in $A^\times$, 
yields an automorphism $\alpha'$ of $A$ in the same class as $\alpha$ 
which extends $\varphi$ and stabilises $L$. If $\alpha$ fixes $P$, so does
$\alpha'$, and hence its restriction to $L$ fixes $P$.
Together with Lemma \ref{ALAut},  this shows that the map sending 
$\alpha$ to the restriction of $\alpha'$ to
$L$ induces  a group  homomorphism $\CT(A)\to$ $\CT(L)$ mapping the 
image of $\Out_P(A)$ in $\CT(A)$ to the image of $\Out_P(L)$ in 
$\CT(L)$, and by Proposition \ref{LtoA} (iv) we have $\Out_P(L)\cong$
$\Hom(E,k^\times)$. 

For the injectivity of this group homomorphism, suppose that 
$\alpha$ stabilises $L$ and restricts to an inner automorphism of $L$. 
By Proposition \ref{LtoA} (iii), the  $A$-$A$-bimodule $A_\alpha$ is 
isomorphic to a direct summand of $A_\alpha \ten_L A$. Since the 
restriction of $\alpha$ to $L$ is inner, we have $A_\alpha\cong$ $A$ 
as $A$-$L$-bimodules. Thus $A_\alpha$ is isomorphic to a direct summand 
of $A\ten_L A$. But then Proposition \ref{LtoA} (iii) implies that 
$A_\alpha\cong$ $A$ as $A$-$A$-bimodules, and hence $\alpha$ is an 
inner automorphism of $A$. This concludes the proof.
\end{proof}

\begin{Proposition} \label{TA-Prop4}
Let $\gamma : L\to$ $A$ be an algebra homomorphism such that
$\gamma(u)=u$ for all $u\in$ $P$ and such that the induced map
$k\tenO L\to k\tenO A$ is the canonical inclusion. Then 
there is an element $c\in 1+\pi A^P$ such that $\gamma(y)=$
${y^c}$ for all $y\in$ $L$.
\end{Proposition}

\begin{proof}
The hypotheses imply that $A$ and $A_\gamma$ are permutation
$P\times P$-modules such that $k\tenO A\cong$ $k\tenO A_\gamma$ 
as $A$-$L$-bimodules, with the isomorphism given by 
$1\otimes a\mapsto 1\otimes a$.
Since $p$-permutation modules over finite group algebras lift
uniquely, up to isomorphism from $k$ to $\CO$, and since 
homomorphisms between $p$-permutation modules lift from $k$
to $\CO$ (see e.g. \cite[Theorem 5.11.2]{LiBookI}) it follows that
there is an $A$-$L$-bimodule isomorphism $A\cong A_\gamma$
lifting the identity map on $k\tenO A$. (Note we must temporarily 
pass to the block algebras to apply the results of \cite{LiBookI}.) 
Consequently this bimodule isomorphism is induced by right 
multiplication on $A$ with an element $c$ in $1+\pi A$. 
Since right multiplication by $c$ is also an isomorphism of right 
$L$-modules, it follows that $c\gamma(y)=$ $yc$ for all $y\in$ $L$.
That is, composing $\gamma$ with the automorphism given by 
conjugating with $c$ gives the inclusion map $L\to$
$A$. Since $\gamma$ fixes $P$, it follows that $c\in$ $A^P$, hence
$c\in$ $1+\pi A^P$,  whence the result.
\end{proof}

\section{Proofs} \label{proofSection}

\begin{proof}[{Proof of Theorem \ref{PicGreencorr}}]
We use the notation from Proposition \ref{LtoA}, as briefly reviewed at the
beginning of Section \ref{SourceAut}.
Let $X$ be  an invertible $B$-$B$-bimodule $X$ representing an element in 
$\CL(B)$. We will show that $X$ corresponds (via the standard Morita
equivalence) to an invertible $A$-$A$-bimodule of the form $A_\alpha$, 
for some algebra automorphism $\alpha$ of $A$ which preserves $L$, or 
equivalently, which restricts to an algebra automorphism $\beta$ of $L$. 
Together with Proposition \ref{Greencorrbimod}, this implies that $L_\beta$
corresponds to the Green correspondent $Y$ of $X$. Before getting into 
details, we show how this completes the proof of Theorem
\ref{PicGreencorr}. Since the Green correspondence is a bijection on the
isomorphism classes of the bimodules under consideration, this shows
that the class of $\beta$ in $\Out(L)$ is uniquely determined
by the class of $\alpha$ in $\Out(A)$ (a fact which follows also directly 
from Lemma \ref{ALAut}), and hence that the map
$\CL(B)\to$ $\CL(C)$ induced by the map $A_\alpha\to$ $L_\beta$
is an injective map. This is a group homomorphism because
for any two algebra automorphisms $\alpha$, $\alpha'$ of
$A$ we have a bimodule isomorphism $A_\alpha\tenA A_{\alpha'}\cong$
$A_{\alpha\circ\alpha'}$. 

We turn now to what remains to be proved, namely that an invertible
$B$-$B$-bimodule $X$ representing an element in $\CL(B)$ corresponds 
to an invertible $A$-$A$-bimodule of the form $A_\alpha$, 
for some algebra automorphism $\alpha$ of $A$ which preserves $L$

As pointed out in \cite[Remark 1.2.(e)]{BoKL}, it follows from 
\cite[Theorem 1.1, Lemma 3.15]{Lifocal} that we have canonical 
isomorphisms 
$$\CL(B) \cong \Hom(P/\foc(\CF),\CO^\times)\rtimes \CT(B)\ ,$$ 
$$\CL(C) \cong \Hom(P/[P,P\rtimes E], \CO^\times) \rtimes \CT(C)\ ,$$
where in the second isomorphism we use the fact that the fusion
system of $L$ is $N_\CF(P)$, which is the same as the fusion 
system of the group $P\rtimes E$ on $P$, and hence its focal
subgroup is $[P,P\rtimes E]$. This is a subgroup of $\foc(\CF)$, and therefore
we may identify $\Hom(P/\foc(\CF),\CO^\times)$ with a subgroup
of $\Hom(P/[P,P\rtimes E],\CO^\times)$. 

It suffices to show separately for $X$ representing an element in
$\CT(B)$ and  in $\Hom(P/\foc(\CF),\CO^\times)$ that $X$ corresponds
to an invertible $A$-$A$-bimodule of the form $A_\alpha$ as above. 
As far as $\CT(B)$ is concerned, this  holds by the Propositions
\ref{TA-Prop1} and \ref{TA-Prop2}. Note that if $\CO=k$, then 
$\Hom(P/\foc(\CF),\CO^\times)$ is trivial, so this concludes 
the proof of Theorem \ref{PicGreencorr} in that case. 

We assume now 
that $\CO$ has characteristic zero (and enough roots of unity, by our 
initial blanket assumption).  Suppose that $X$  represents an  element in 
the canonical image of $\Hom(P/\foc(\CF),\CO^\times)$ in $\CL(B)$.
We need to show that  then $X$ corresponds to an invertible 
$A$-$A$-bimodule of the form $A_\alpha$ as stated above. 
Note that the image of the group $\Hom(P/\foc(\CF), \CO^\times)$ in 
$\CL(B)$ is equal to the kernel of the canonical map $\CL(B)\to$
$\CT(k\tenO B)$; this follows from \cite[Remark 1.2.(d)]{BoKL}. This
kernel consists only of isomorphism classes of linear source invertible 
$B$-$B$-bimodules with diagonal vertex $\Delta P=$ 
$\{(u,u)\ |\ u\in P\}$. More precisely, by \cite[Lemmas 2.3, 2.7]{BoKL}, 
if $X$ is an invertible $B$-$B$-bimodule with a linear source, then 
there is a unique group homomorphism $\zeta : P\to$ $\CO^\times$
such that $\foc(\CF)\leq \ker(\zeta)$, and such that $X$ is isomorphic
to a direct summand of $\OG i_\eta \tenOP i\OG$, where $\eta$ is the
algebra automorphism of $\OP$ given by $\eta(u)=$ $\zeta(u) u$ for all
$u\in $ $P$. By the results in \cite[Section 3]{Lifocal}, $\eta$ extends to
an algebra automorphism $\alpha$ of $A$ which induces the identity on 
$k\tenO A$, and through the canonical Morita equivalence, $X$
corresponds to the $A$-$A$-bimodule $A_\alpha$. 
By Proposition \ref{TA-Prop4} applied to $\alpha^{-1}$ restricted to $L$, 
we may choose $\alpha$ such that it stabilises $L$. Thus the
restriction of $\alpha$ to $L$ yields an element $L_\alpha$ whose
isomorphism class belongs to the image of 
$\Hom(P/[P,P\rtimes E],\CO^\times)$ in $\CL(L)$.
\end{proof}

\begin{proof}[{Proof of Theorem \ref{Picbound}}]
We use the notation introduced in Section \ref{backgroundSection}.
The key ingredient is the exact sequence of groups
$$\xymatrix{1 \ar[r] & \Out_P(A) \ar[r]& \CE(B) \ar[r]^(.32){\Phi} &
D(P,\CF)\rtimes \Out(P,\CF) }$$
from \cite[Theorem 1.1]{BoKL} (we write here $D(P,\CF)$ 
instead of $D_\CO(P,\CF)$). Denote by $\CE^\Delta(B)$ the subgroup of 
all elements in $\CE(B)$ whose image in $D(P,\CF)\rtimes \Out(P,\CF)$ is 
of the form $(V, \Id_P)$, for some element $V$ in $D(P,\CF)$. That is, 
$\CE^\Delta(B)$ is the inverse image in $\CE(B)$ under the map $\Phi$ of 
the normal subgroup $D(P,\CF)$ of $D(P,\CF)\rtimes \Out(P,\CF)$. Thus 
$\CE^\Delta(B)$ is a normal subgroup of $\CE(B)$ such that the quotient
$\CE(B)/\CE^\Delta(B)$ is isomorphic to a subgroup of $\Out(P,\CF)$. In 
particular, the order of this quotient is determined in terms of the 
defect group $P$. Now $\CE(B)$ is a finite group (by 
\cite[Theorem 1.1 (iii)]{BoKL}), and hence $\CE^\Delta(B)$ is finite. 
Thus the image of $\CE^\Delta(B)$ under $\Phi$ in $D(P,\CF)$ is finite, 
hence contained in the torsion subgroup of the Dade group $D(P)$. By
\cite[Corollary 2.4]{PuFeit}, the torsion subgroup of $D(P)$ is finite,
hence a finite invariant of $P$, and so the image of $\CE^\Delta(B)$ in 
$D(P)$ is bounded in terms of $P$. By 
\cite[Proposition 14.9]{Puigmodules}, the kernel $\Out_P(A)$ of $\Phi$ 
is isomorphic to a subgroup of $\Hom(E,k^\times)$, where $E\cong$ 
$\Out_\CF(P)$ is the inertial quotient. Thus $\ker(\Phi)$ is also 
bounded in terms of $P$. The result follows.
\end{proof}

For the proof of Theorem \ref{Picbound2} we need the following 
observations; we use the well-known fact that the Cartan matrix of
a split finite-dimensional $k$-algebra $A$ is of the form 
$(\dim_k(iAj))$, where $i$, $j$ run over a set of representatives of the 
conjugacy classes of primitive idempotents in $A$ (see e.g. 
\cite[Theorem 4.10.2]{LiBookI}).

\begin{Lemma} \label{Picbound2Lemma}
Let $A$ be a split finite-dimensional $k$-algebra. Let $I$ be a set of
representatives of the conjugacy classes of primitive idempotents in 
$A$. For $i\in I$ set $S_i=Ai/J(A)i$, and for $i$, $j\in I$ set $c_{ij}=$
$\dim_k(iAj)$. Let $M$ be an invertible $A$-$A$-bimodule. Denote 
by $\pi$ the unique permutation of $I$ satisfying $S_{\pi(i)}\cong$
$M\tenA S_i$ for all $i\in I$. We have
$$\dim_k(M) = \sum_{i,j\in I}\ \ c_{ij}\dim_k(S_{\pi(i)}) \dim_k(S_j)\ .$$
Moreover, for any $i$, $j\in$ $I$, we have $c_{\pi(i)\pi(j)}=c_{ij}$.
\end{Lemma}

\begin{proof}
Since $A$ is split, for any $i\in$ $I$ we have $iS_i\cong$ $k$, and for
any two different $i$, $j\in$ $I$ we have $jS_i=0$.
As a right $A$-module, $M$ is a progenerator, and hence we have an 
isomorphism of right $A$-modules $M\cong$ $\oplus_{i\in I}\ (iA)^{m_i}$
for some positive integers $m_i$. Thus we have an isomorphism of vector
spaces $M\tenA S_i \cong$ $\oplus_{j\in I} (jA\tenA S_i)^{m_j}$. By the
above, the terms with $j\neq i$ are zero while $iA\tenA S_i$ is 
one-dimensional, and hence $\dim_k(S_{\pi(i)})=$ $\dim_k(M\tenA S_i)=$
$m_i$. Note that $\dim_k(iA)=$ $\sum_{j\in I}\ c_{ij} \dim_k(S_j)$. Thus
$$\dim_k(M)=\sum_{i\in I} \dim_k(iA)\cdot m_i=\sum_{i,j\in I}\ \ c_{ij}\dim_k(S_{\pi(i)}) \dim_k(S_j)$$
as stated. 
Since the functor $M\tenA -$ is an equivalence sending $S_i$ to a module
isomorphic to $S_{\pi(i)}$, it follows that this functor sends $Ai$ to
a module isomorphic to $A\pi(i)$ and induces
isomorphisms $\Hom_A(Ai, Aj)\cong$ $\Hom_A(A\pi(i),A\pi(j))$,
hence $iAj\cong$ $\pi(i)A\pi(j)$. The equality $c_{\pi(i)\pi(j)}=c_{ij}$
follows.
\end{proof}

\begin{proof}[{Proof of Theorem \ref{Picbound2}}]
In order to prove Theorem \ref{Picbound2} we may assume that $\CO=k$.
We use the notation as in Lemma \ref{Picbound2Lemma}. By the 
assumptions, the Cartan matrix $C=$ $(c_{ij})_{i,j\in I}$ of $A$ is
symmetric and positive definite. Thus the map $(x,y)\to$ $x^TCy$ from
$\R^{|I|}\times \R^{|I|}$ to $\R$ is an inner product. The Cauchy-Schwarz
inequality yields $|x^TCy|^2\leq |x^T Cx| \cdot |y^TCy|$. We are going to
apply this to the dimension vectors $x=(\dim_k(S_i))_{i\in I}$ and
$y=(\dim_k(S_{\pi(i)}))_{i\in I}$. By Lemma \ref{Picbound2Lemma},
we have $\dim_k(M)=$ $x^TCy$. Applied to $M=A$ (and $\pi=\Id$) we also
have that $\dim_k(A)=$ $x^T Cx$. The last statement in Lemma
\ref{Picbound2Lemma} implies that $x^T Cx=$ $y^TCy$. Thus the
Cauchy-Schwarz inequality yields $\dim_k(M)=$ $x^T Cy\leq$ $x^TCx=$
$\dim_k(A)$ as stated. 

The Cauchy-Schwarz inequality is an equality if and only if the dimension
vectors $x$ and $y$ are linearly dependent. Since both vectors consist of 
the same positive integers (in possibly different orders) this is the case if 
and only if $x=y$, or equivalently, if and only if $\dim_k(S_{\pi(i)})=$ 
$\dim_k(S_i)$  for all $i\in I$. By \cite[Proposition 4.7.18]{LiBookI},
this  holds if and only if there is an $A$-$A$-bimodule isomorphism 
$M\cong$ $A_\alpha$ for some $\alpha\in$ $\Aut(A)$. 
This completes the proof.
\end{proof}

\begin{proof}[{Proof of Theorem \ref{LPicnormal}}]
It is clearly enough to prove the theorem for $\CO=$ $k$. Let $G$ be a 
finite group and $B=$ $kGb$ a block of $kG$ with normal defect group 
$P$. We first claim that we may assume $B$ is isomorphic to its own 
source algebra. More specifically we assume that $G=$ $P\rtimes H$, 
for a $p'$-group $H$ and $Z=C_H(P)$ a cyclic subgroup such that $b\in$ 
$kZ$. Indeed $B$ is certainly source algebra equivalent (or equivalent 
as an interior $P$-algebra) to such a block (see e.g. 
\cite[Theorem 6.14.1]{LiBookII}). Since, by \cite[Lemma 2.8(ii)]{BoKL}, 
source algebra equivalences preserve $\CE(B)$ and $\CL(B)$, we may 
assume that $B$ is of the desired form.

Let $M$ be an invertible $B$-$B$-bimodule with endopermutation source. 
Let $Q$ be a vertex of $M$ which, by \cite[Lemma 1.1]{BoKL}, is 
necessarily of the form $\Delta\varphi=$ 
$\{(\varphi(u),u)\ |\ u\in P\}\leq$ $P\times P$, for some $\varphi\in$ 
$\Aut(P)$. Let $V$ be a source for $M$ with respect to the vertex $Q$. 
In particular, $V$ is absolutely indecomposable (see e.g. 
\cite[Proposition 7.3.10]{LiBookII}). It follows from Green's Indecomposablity 
Theorem that $U=$ $\Ind_Q^{P\times P}(V)$ is indecomposable. We now 
consider $I=\Stab_{G\times G}(U)$, the stabiliser of $U$ in $G \times G$. 
If $h\in$ $H$, then the $\CO(P\times P)$-module ${^{(h,1)}U}$ has vertex 
$\Delta(c_h\circ \varphi)$, where $c_h$ denotes conjugation by $h$. 
Therefore, if $(h,1)\in$ $I$, then
$$\Delta(c_h\circ \varphi)=(x,y)(\Delta\varphi)(x,y)^{-1}=
\Delta(c_x\circ \varphi\circ c_{y^{-1}}),$$
for some $x,y\in$ $P$. In other words,
$$c_h=c_x\circ \varphi\circ c_{y^{-1}}\circ\varphi^{-1}=
c_x\circ c_{\varphi(y)^{-1}}.$$
In particular, $c_h$ is an inner automorphism of $P$. However, since 
$H/Z$ is a $p'$-group, $c_h$ is an inner automorphism of $P$ if and only 
if $h\in$ $Z$. So we have an injective map between sets 
$H/Z\to (G\times G)/I$, giving that $[G\times G:I]\geq$ $[H:Z]$. 
Therefore, since $M\cong$ $\Ind_I^{G\times G}(W)$, for some direct 
summand $W$ of $\Ind_{P\times P}^I(U)$,
$$\dim_k(M)= [G\times G:I]\cdot\dim_k(W)\geq 
[H:Z]\cdot\dim_k(W)\geq[H:Z]\cdot\dim_k(U)$$
$$=[H:Z]\cdot[P\times P:Q]\cdot\dim_k(V)=[H:Z]\cdot |P|\cdot\dim_k(V)=[G:Z]\cdot\dim_k(V)=\dim_k(B)\cdot\dim_k(V).$$
By Theorem \ref{Picbound2}, we have $\dim_k(B)\geq$ $\dim_k(M)$.
This yields that $\dim_k(V)=1$ as desired.
\end{proof}

\begin{Example} \label{EPicExamples}
Let $p=$ $3$ and $G=$ $Q_8\rtimes P$, where the action of $P\cong$ 
$C_3$ on $Q_8$ is non-trivial. Set $B=$ $\CO Gb$, where $b=(1-z)/2$ and 
$z$ is the unique non-trivial central element in $Q_8$. Now 
$\CO Q_8 b\cong$ $M_2(\CO)$ and conjugation by a non-trivial element 
of $P$ induces an non-trivial automorphism of $\CO Q_8 b$. We 
temporarily assume $\CO=k$. Since every $\CO$-algebra automorphism 
of $M_2(\CO)$ is inner and every non-trivial element of order three in
$M_2(\CO)$ is conjugate to
$$x=\left( {\begin{array}{cc}
   1 & 1 \\
   0 & 1 \\
  \end{array} } \right)\ ,$$
it follows that
$(\CO Q_8 b)^P\cong$ $\CO\cdot 1\oplus \CO\cdot x\subset$ $M_2(\CO)$. In 
particular, $(\CO Q_8 b)^P$ is local and so $B^P\cong$ 
$(\CO Q_8 b)^P\tenO\CO P$ is also local meaning $B$ is its own 
source algebra. This holds even with the assumption that $\CO=$ $k$ dropped.

As $B$ is nilpotent, we can apply \cite[Example 7.2]{BoKL} and construct an 
element of $\CE(B\tenO \CO P)$ not in $\CL(B\tenO \CO P)$. 
With Theorem \ref{LPicnormal} in mind, we have shown that 
Theorem \ref{PicGreencorr} does not hold with $\CL$ replaced by $\CE$.
\end{Example}

\begin{Remark} \label{BoKLcorrection}
In view of the notational conventions in \cite[Proposition 2.6]{BoKL},
the group homomorphism $\Phi$ in \cite[Theorem 1.1]{BoKL} should send
$X$ to $(V, \varphi^{-1})$ instead of $(V,\varphi)$.
\end{Remark}

\medskip\noindent
{\it Acknowledgement.} The authors would like to thank Charles Eaton and
Radha Kessar for some very helpful conversations during the writing of this
paper. The second author acknowledges support from EPSRC, grant 
EP/T004606/1.



\begin{thebibliography}{WWW}

\bibitem{AOV} K. Andersen, B. Oliver, and J. Ventura, {\em Reduced,
tame, and exotic fusion systems.} Proc. London Math. Soc. {\bf 105}
(2012), 87--152.


\bibitem{BoKL} R. Boltje, R. Kessar, and M. Linckelmann, {\em On Picard 
groups of blocks of finite groups.} Preprint, arXiv:1805.08902,
J. Algebra, doi: 10.1016/j.jalgebra.2019.02.045


\bibitem{CR2} C. W. Curtis and I. Reiner, {\it Methods of
Representation theory} Vol. II, John Wiley and Sons,
New York, London, Sydney (1987).

\bibitem{EatLivPic} C. Eaton and M. Livesey, {\em Some examples 
of Picard groups of blocks.} Preprint, arXiv:1810.10950,
J. Algebra, doi: 10.1016/j.jalgebra.2019.08.004

\bibitem{EisPic} F. Eisele, {\em On the geometry of lattices and 
finiteness of Picard groups.} arXiv:1908.00129 (2019).  

\bibitem{FanPuig} Y. Fan, L. Puig, {\em On blocks with nilpotent
coefficient extensions}, Alg. Repr. Theory {\bf 2} (1999)

\bibitem{HeKi} M. Hertweck and W. Kimmerle, {\em On principal
blocks of $p$-constrained groups}, Proc. London Math. Soc. {\bf 84}
(2002), 179--193.

\bibitem{Kuenormal} B. K\"ulshammer, {\em Crossed products and blocks
with normal defect groups}, Comm. Algebra {\bf 13} (1985), 147--168.

\bibitem{LiKlein} M. Linckelmann, {\em The source algebras of blocks with
a Klein four defect group.} J. Algebra {\bf 167} (1994), 821--854.

\bibitem{LiBookI} M. Linckelmann, {\em The block theory of finite
groups I.} London Math. Soc. Student Texts {\bf 91}, Cambridge
University Press, 2018.

\bibitem{LiBookII} M. Linckelmann, {\em The block theory of finite
groups II.} London Math. Soc. Student Texts {\bf 92}, Cambridge
University Press, 2018.

\bibitem{Lifocal} M. Linckelmann, {\em On automorphisms and focal 
subgroups of blocks}. Geometric and topological aspects of the 
representation theory  of finite groups, 235--249, Springer Proc. 
Math. Stat., {\bf 242}, Springer, Cham, 2018.

\bibitem{LiSchroll} M. Linckelmann and S. Schroll, {\em On the Coxeter
complex and Alvis-Curtis duality for principal $\ell$-blocks of
$\GL_n(q)$.} J. Algebra Appl. {\bf 4} (2005), 225--229.

\bibitem{Puigpoint} L. Puig, Pointed groups and construction of
characters. \emph{ Math. Z.} {\bf 176} (1981), 265--292.

\bibitem{Puigloc} L. Puig, {\em Local fusion in block source
algebras}, J. Algebra {\bf 104} (1986), 358--369.

\bibitem{Puigmodules} L. Puig, {\em Pointed groups and construction of
modules}, \emph{J. Algebra} {\bf 116} (1988), 7--129.

\bibitem{PuFeit} L. Puig, {\em Affirmative answer to a question of Feit},
J. Algebra  {\bf 131}  (1990),  no. 2, 513--526.

\bibitem{Puigbook} L. Puig, {\em On the local structure of Morita
and Rickard equivalences between Brauer blocks}, Progress in
Math. {\bf 178}, Birkh\"auser Verlag, Basel (1999)

\bibitem{RoSc} K. W. Roggenkamp and L. L. Scott, {\em Isomorphisms 
of $p$-adic group rings.} Annals Math. {\bf 126} (1987), 593--647.

\bibitem{Scottnotes} L. L. Scott, unpublished notes (1990).


\end{thebibliography}
\end{document}